\documentclass[12pt]{amsart}
\usepackage{amsmath,amsfonts,amssymb,amsthm}

\theoremstyle{plain}
\newtheorem{theorem}{Theorem}

\theoremstyle{definition}

\newtheorem{remark}[theorem]{Remark}

\newcommand{\Z}{{\mathbb Z}}
\newcommand{\R}{{\mathbb R}}

\newcommand{\C}{{\mathbb C}}

\begin{document}

\title{An asymptotically sharp form of Ball's integral inequality}
\author{
 R.~Kerman, R.~Ol\!'hava, S.~Spektor}

\address{ Ron Kerman,
\noindent Address: Brock University, Canada}
 \email{ rkerman@brocku.ca}

 \address{ Rastislav Ol\!'hava,
 \noindent Address: Charles University, Prague}
 \email{ olhavara@centrum.sk}

\address{ Susanna Spektor,
\noindent Address: University of Alberta, Canada}
\email{ spektor@ualberta.ca}

\thanks{The research of the second author was in part supported by grant no. 13-14743S of the Grant Agency of the Czech Republic. }
\date{}
\maketitle

\thispagestyle{empty}

\begin{abstract}
 We solve the open problem of determining the second order term in the asymptotic expansion of the integral in Ball's integral inequality. In fact, we provide a method by which one can compute any term in the expansion. We also indicate how to derive an asymptotically sharp form of a generalized Ball's integral inequality.
\end{abstract}

\setcounter{page}{1}


\section{\textbf{Introduction}}

To prove that every $(n-1)$-dimensional section of the unit cube in $\R^n$ has volume at most $\sqrt 2$, K.~Ball \cite{Ball:1986} made essential use of the  inequality
\begin{equation}\label{1}
{\sqrt n}\int_{-\infty}^{\infty} \left|\frac{\sin t}{t}\right|^ndt\leq {\sqrt 2}{\pi}, \quad  n\geq 2,
\end{equation}
 in which equality holds if and only if $n=2$.

\bigskip

Later, Ball's integral inequality (\ref{1}) was proved  using different methods; see \cite{BBL:2010, NP:2000} (also see  \cite{KK} for an analogue of Ball's inequality). Independently of Ball,  D.~Borwein, J.~M.~Borwein and I.~E.~Leonard investigated, in \cite{BBL:2010},   the asymptotic expansion of the left side of (\ref{1}). They established the existence of real constants, $c_j$, such that
\begin{equation}\label{2}
{\sqrt n}\int_{0}^{\infty} \left|\frac{\sin t}{t}\right|^ndt\sim \sqrt{\frac{3 \pi}{2}}-\frac{3}{20}\sqrt{\frac{3 \pi}{2}}\frac 1n+\sum_{j=2}^{\infty}\frac{c_j}{n^j}, \quad \textit{as} \quad n\longrightarrow \infty,
\end{equation}
and posed the problem of determining the value of $c_2$.

\bigskip

K.~Oleszkiewicz and A.~Pelczy$\acute{\textmd{n}}$ski, in \cite{OP:2000}, proved the following variant of Ball's inequality, namely,
\begin{equation}\label{3}
n\int_{0}^{\infty}\left(\frac{2|J_1(t)|}{t}\right)^ntdt\leq 4, \quad n\geq 2,
\end{equation}
involving a special case of
\begin{equation*}
J_{\nu}(t):=\sum_{j=0}^{\infty}(-1)^j \left(\frac{t}{2}\right)^{2j+{\nu}}\frac{1}{j!\Gamma(j+{\nu}+1)}, \quad t\geq 0, {\nu}\geq \frac 12,
\end{equation*}
the Bessel function of order ${\nu}$. They showed that, with the method used to establish their inequality (\ref{3}), one can prove (\ref{1}). Also, they discussed the more general inequality
\begin{equation*}
n^{\nu}\int_0^{\infty}\left(2^{\nu}\Gamma({\nu}+1)\frac{\left|J_{\nu}(t)\right|}{t^{\nu}}\right)^nt^{2{\nu}-1}dt
< {2^{\nu}}\left(\int_0^{\infty}\left(2^{\nu}\Gamma({\nu}+1)\frac{J_{\nu}(t)}{t^{\nu}}\right)^2t^{2{\nu}-1}dt\right), \quad n>2.
\end{equation*}
They conjectured it holds  if and only if $\displaystyle{\frac 12 \leq \nu \leq 1}$. In this connection, they pointed out that H.~K$\ddot{\textmd{o}}$nig
has noticed the inequality is false when $\displaystyle{{\nu}=\frac k2,k=3,4,\ldots .}$

\bigskip

Our  paper is divided into three sections  and an appendix.
The first section is an Introduction.
The second section is devoted to  calculating the $c_2$ in (\ref{2}), thereby solving the problem posed by D.~Borwein, J.~M.~Borwein and I.~E.~Leonard.
The method that gives $c_2$  can be used to derive \textit{any} term in the asymptotic expansion in  (\ref{2}).
In the third section, we indicate how the method of Section 2 enables one to determine the asymptotic expansion of
\begin{equation*}
n^{\nu}\int_0^{\infty}\left(2^{\nu}\Gamma({\nu}+1)\frac{\left|J_{\nu}(t)\right|}{t^{\nu}}\right)^nt^{2{\nu}-1} dt, \quad n\geq 2,
\end{equation*}
for all ${\nu\geq 1/2}$.

\bigskip

\section{\textbf{An Asymptotically Sharp Form of Ball's Integral Inequality}}

In this section we answer the open question of D.~Borwein, J.~M.~Borwein and I.~E.~Leonard in the following theorem.
\begin{theorem}\label{t.h.A}
Let
\begin{equation*}
I(n):={\sqrt n}\int_{0}^{\infty} \left|\frac{\sin t}{t}\right|^ndt,
\end{equation*}
$n\geq 2$, and fix $m \in \Z_+, m\geq 3$. Then, there exist constants $c_3, c_4, \ldots, c_m$ such that
\begin{equation*}
I(n)=\sqrt{\frac{3\pi}{2}}\left[1-\frac{3}{20}\frac 1n - \frac{13}{1120} \frac{1}{n^2}+\sum_{j=3}^m \frac{c_j}{n^j}\right]+O\left(\frac{1}{n^{m+1}}\right).
\end{equation*}
\end{theorem}
\begin{proof}
We first observe that $I(n)$ can be replaced by
\begin{align*}
J(n):=\sqrt n \int_0^{\sqrt 6}\left|\frac{\sin t}{t}\right|^ndt.
\end{align*}
Indeed,
\begin{align*}
{\sqrt n}\int_{\sqrt 6}^{\infty} \left|\frac{\sin t}{t}\right|^ndt\leq \sqrt n \int_{\sqrt 6}^{\infty}t^{-n}dt=\frac{\sqrt{6n}}{n-1}6^{-n/2}.
\end{align*}

Next, set
\begin{align*}
T_k(t):=\sum_{j=0}^k \frac{(-1)^j}{(2j+1)!}t^{2j}.
\end{align*}
Then, for $k$ odd, one has
\begin{align*}
0\le T_k(t)\leq \frac{\sin t}{t}\leq T_{k+1}(t),
\end{align*}
$t\in (0, \sqrt 6)$, so,
\begin{align*}
\int_0^{\sqrt 6}T_k(t)^{n}dt \leq \int_0^{\sqrt 6}\left(\frac{\sin t}{t}\right)^n dt \leq \int_0^{\sqrt 6}T_{k+1}(t)^{n} dt.
\end{align*}

Therefore, it suffices to show there exists constants $c_3, c_4, \ldots, c_m$ such that
\begin{align*}
K(n):=\sqrt n \int_0^{\sqrt 6}T_k(t)^n dt=\sqrt{\frac{3\pi}{2}}\left[1-\frac{3}{20}\frac 1n - \frac{13}{1120} \frac{1}{n^2}+\sum_{j=3}^m \frac{c_j}{n^j}\right]+O\left(\frac{1}{n^{m+1}}\right),
\end{align*}
whenever $k\geq m+1$.

Making the change of variable $\displaystyle{s=\frac{t}{\sqrt{n}}}$ in $\displaystyle{\int_0^{\sqrt 6}T_k(s)^n}ds$ we obtain
\begin{align*}
K(n)= \int_0^{\sqrt {6n}}T_k\left(\frac{t}{\sqrt n}\right)^ndt=\int_0^{\sqrt{6n}}e^{-t^2/6}\left[e^{t^2/{6n}}T_k\left(\frac{t}{\sqrt n}\right)\right]^n dt.
\end{align*}

Now,
\begin{align*}
e^{t^2/{6n}}=\sum_{j=0}^{\infty}\frac{\left(\frac{t^2}{6n}\right)^j}{j!},
\end{align*}
whence
\begin{align*}
e^{t^2/{6n}}T_k\left(\frac{t}{\sqrt n}\right)=1+\sum_{j=2}^{\infty}\frac{a_j}{n^j}t^{2j},
\end{align*}
in which
\begin{align*}
a_j=\sum_{i=0}^{\min[j,k]}\frac{1}{i!6^i}\frac{(-1)^{j-i}}{(2(j-i)+1)!}.
\end{align*}
Using Newton's Binomial Formula we obtain
\begin{align}\label{5}
\left[e^{t^2/{6n}}T_k\left(\frac{t}{\sqrt n}\right)\right]^{n}
&=1+n\left[\sum_{j=2}^{\infty}\frac{a_j}{n^j}t^{2j}\right]+\frac{n(n-1)}{2}\left[\sum_{j=2}^{\infty}\frac{a_j}{n^j}t^{2j}\right]^2+ \ldots \notag\\
&+\frac{n(n-1)\ldots (n-m+1)}{m!}\left[\sum_{j=2}^{\infty}\frac{a_j}{n^j}t^{2j}\right]^m+ \ldots .
 \end{align}
 We observe that, for  $t \in [0, \sqrt{6n}]$,
 \begin{align*}
 \left|\sum_{j=2}^{\infty}\frac{a_j}{n^j}t^{2j}\right|=\left|e^{\frac{t^2}{6n}}T_k\left(\frac{t}{\sqrt n}\right)-1 \right|<1.
 \end{align*}

Only the first $m+1$ terms on the right-hand side of (\ref{5}) yield the powers $\displaystyle{\frac{1}{n^0}, \frac{1}{n}, \frac{1}{n^2}, \ldots, \frac{1}{n^m}}$. We get
\begin{align*}
 \hspace{0.1cm}
\frac{1}{n^0}, \hspace{0.3cm} \frac{a_2t^4}{n}=-\frac{t^4}{180n}, \hspace{0.3cm} \left(a_3t^6+\frac 12 a_2^2 t^8\right)\frac{1}{n^2}= \left(-\frac{t^6}{2835} +\frac{t^8}{64800}\right)\frac{1}{n^2} \hspace{0.1cm}
\end{align*}
and so on.

The highest power of $t$ yielding $\displaystyle{\frac{1}{n^m}}$ is $\displaystyle{t^{4m}}$. Accordingly, we write
\begin{align}\label{6}
\left[e^{t^2/{6n}}T_k\left(\frac{t}{\sqrt n}\right)\right]^{n}=\sum_{j=0}^{2m}b_j\left(\frac{t}{\sqrt n}\right)^{2j}+R_{2m}\left(\frac{t}{\sqrt n}\right), \quad b_j=b_j(n),
\end{align}
in which
\begin{align*}
\left|R_{2m}\left(\frac{t}{\sqrt n}\right)\right|\leq C \frac{t^{4m+2}}{n^{2m+1}},
\end{align*}
the constant $C>1$ being independent of $t \in [0, \sqrt{6n}]$.

For concreteness, we now work with the polynomial of degree $28$ in (\ref{6}) corresponding to $m=7$. It is given in the Appendix. Formula (\ref{6}) becomes
\begin{align}\label{7}
\left[e^{t^2/{6n}}T_8\left(\frac{t}{\sqrt n}\right)\right]^{n}=\sum_{j=0}^{14}b_j\left(\frac{t}{\sqrt n}\right)^{2j}+O\left(\frac{t^{30}}{ n^{15}}\right), \quad b_j=b_j(n),
\end{align}
and gives all the correct terms in the asymptotic expansion up to $\displaystyle{\frac{1}{n^7}}$. Multiplying the polynomial in (\ref{7}) by $\displaystyle{e^{-t^2/6}}$, integrating the product from $0$ to $\sqrt{6n}$ and using the fact that
\begin{align*}
\int_0^{\sqrt{6n}}e^{-t^2/6}t^{2j}dt&=6^{j+\frac 12}\int_0^{\infty}e^{-t^2}t^{2j}dt+O\left(\frac{1}{n^{8}}\right)\\
&=3^j(2j-1)(2j-3)\ldots 1\, \sqrt{\frac{3 \pi}{2}}+O\left(\frac{1}{n^{8}}\right),
\end{align*}
$j=1, 2, \ldots, 2m$, we obtain, with an error of $\displaystyle{O\left(\frac{1}{n^{8}}\right)}$,
\begin{align*}
K(n)=\sqrt{\frac{3 \pi}{2}}\left[1-\frac{3}{20}\frac 1n-\frac{13}{1120}\frac{1}{n^2}+\sum_{j=3}^7\frac{c_j}{n^j}\right].
\end{align*}
\end{proof}
\begin{remark}
Working with the polynomial of degree $28$ in the Appendix one can show
\begin{align*}
&c_3=\frac{27}{3200}, \quad c_4=\frac{52791}{3942400}, \quad c_5=-\frac{5270328789}{136478720000},\\
 &\quad\quad c_6=-\frac{124996631}{10035200000}, \quad c_7=-\frac{5270328789}{136478720000}.
\end{align*}
\end{remark}
\begin{remark}
A proof using splines that $I(n)\sim \displaystyle{\sqrt{\frac{3 \pi}{2}}}$ is given in \cite{KermanSpektor:2012}.
\end{remark}

\section{\textbf{A generalized Ball's integral inequality}}

We indicate how to determine constants $c_0, c_1, c_2, c_3, \ldots, c_m$ so that, with $n\geq 2$,
\begin{align}\label{9}
I_{\nu}(n):=n^{\nu}\int_0^{\infty}\left(\frac{2^{\nu}\Gamma({\nu}+1)|J_{\nu}(t)|}{t^{\nu}}\right)^nt^{2{\nu}-1}dt=c_0+\frac{c_1}{n}+\frac{c_2}{n^2}+\frac{c_3}{n^3}+\ldots+\frac{c_m}{n^m}+O\left(\frac{1}{n^{m+1}}\right). \end{align}
For definiteness, we do this when $m=3$ .

Our first observation is that $I_{\nu}(n)$ may be replaced by
\begin{align}\label{10}
n^{\nu}\int_0^{2^{\nu}\Gamma({\nu}+1)}\left(\frac{2^{\nu}\Gamma({\nu}+1)|J_{\nu}(t)|}{t^{\nu}}\right)^nt^{2{\nu}-1}dt.
\end{align}
Indeed, using the estimate
\begin{align*}
|J_{\nu}(t)|\leq ct^{-\frac 13}, \quad t \in \R_+, {\nu}\geq 1, c=0.7857468704\ldots,
\end{align*}
given in \cite{Landau:2000}, we get, for $n$ sufficiently large,
\begin{align*}
n^{\nu}\int_x^{\infty}\left(\frac{2^{\nu}\Gamma({\nu}+1)|J_{\nu}(t)|}{t^{\nu}}\right)^nt^{2{\nu}-1}dt&\leq n^{\nu}\int_x^{\infty}\left(2^{\nu}\Gamma({\nu}+1)ct^{-{\nu}-\frac 13}\right)^nt^{2{\nu}-1}dt\\
&=n^{\nu}\left(2^{\nu}\Gamma({\nu}+1)c\right)^n\int_x^{\infty}t^{-\left({\nu}+\frac 13\right)n+2{\nu}-1} dt\\
&=\frac{n^{\nu}\left(2^{\nu}\Gamma({\nu}+1)c\right)^nx^{-\left({\nu}+\frac 13\right)n+2{\nu}}}{n\left({\nu}+\frac 13\right)-2{\nu}}\\
&\leq n^{{\nu}-1}c^{\nu},
\end{align*}
with $x=2^{\nu}\Gamma({\nu}+1)$.

As we did in  Section 2 for $\displaystyle{\frac{\sin t}{t}}$, we approximate $2^\nu t^{-\nu}\Gamma({\nu}+1)J_{\nu}(t)$ in (\ref{10}) by the $k$-th partial sum of its Maclaurin series, namely,
\begin{align}\label{11}
T_k(t):=\sum_{j=0}^k \frac{\left(-\frac {t^2}{4}\right)^j\Gamma({\nu}+1)}{j!\Gamma({\nu}+j+1)},
\end{align}
where $k \geq m+1$.

The change of variable $\displaystyle{t\longrightarrow\frac{t}{\sqrt n}}$ in the integral of
\begin{align*}
K_{\nu}(n):=n^{\nu}\int_0^{2^{\nu}\Gamma({\nu}+1)}T_k(t)^nt^{2{\nu}-1} dt
\end{align*}
yields
\begin{align*}
\int_0^{2^{\nu}\Gamma({\nu}+1)\sqrt n}T_k\left(\frac{t}{\sqrt n}\right)^n t^{2{\nu}-1} dt.
\end{align*}
Using the Maclaurin expansion of $\displaystyle{\exp\left(\frac{t^2}{4n({\nu}+1)}\right)}$, together with (\ref{11}), we obtain
\begin{align*}
K_{\nu}(n)&=\int_0^{2^{\nu}\Gamma({\nu}+1)\sqrt n}\exp\left(\frac{-t^2}{4({\nu}+1)}\right)\left[\exp\left(\frac{t^2}{4n({\nu}+1)}\right)T_k\left(\frac{t}{\sqrt n}\right)\right]^nt^{2{\nu}-1} dt\\
&=\int_0^{2^{\nu}\Gamma({\nu}+1)\sqrt n}\exp\left(\frac{-t^2}{4({\nu}+1)}\right)\left[1+\sum_{j=2}^{\infty}a_j\left(\frac{t^2}{4n}\right)^j\right]^nt^{2{\nu}-1} dt\\
&=\int_0^{2^{\nu}\Gamma({\nu}+1)\sqrt n}\exp\left(\frac{-t^2}{4({\nu}+1)}\right)\Bigg(1+n\left[\sum_{j=2}^{\infty}a_j\left(\frac{t^2}{4n}\right)^j\right]+ \ldots\\
&+\frac{n(n-1)\ldots(n-m+1)}{m!}\left[\sum_{j=2}^{\infty}a_j\left(\frac{t^2}{4n}\right)^j\right]^m+\ldots
\Bigg)t^{2{\nu}-1} dt,
\end{align*}
in which
\begin{align*}
a_j=\sum_{i=0}^{\min[j, k]}(-1)^i\frac{\Gamma({\nu}+1)}{({\nu}+1)^{j-1}(j-i)!i!\Gamma({\nu}+i+1)}.
\end{align*}

One finds that
\begin{align*}
a_2&=\frac{-1}{2({\nu}+1)^2({\nu}+2)}\\
a_3&=\frac{-2}{3({\nu}+1)^3({\nu}+2)({\nu}+3)}\\
a_4&=\frac{{\nu}-5}{8({\nu}+1)^4({\nu}+2)({\nu}+3)({\nu}+4)}
\end{align*}
and that, moreover,
\begin{align*}
c_0&=\int_0^{\infty}\exp\left(\frac{-t^2}{4({\nu}+1)}\right)t^{2{\nu}-1}dt=\frac{4^{\nu}}{2}({\nu}+1)^{\nu}\Gamma({\nu})\\
c_1&=\frac{a_2}{16}\int_0^{\infty}\exp\left(\frac{-t^2}{4({\nu}+1)}\right)t^4t^{2{\nu}-1}dt=\frac{-4^{{\nu}-1}({\nu}+1)^{{\nu}}\Gamma({\nu}+2)}{{\nu}+2}\\
c_2&=\frac{a_3}{64}\int_0^{\infty}\exp\left(\frac{-t^2}{4({\nu}+1)}\right)t^6t^{2{\nu}-1}dt
+\frac{a_2^2}{512}\int_0^{\infty}\exp\left(\frac{-t^2}{4({\nu}+1)}\right)t^8t^{2{\nu}-1}dt\\
&\hspace{6.5cm}=4^{{\nu}-2}({\nu}+1)^{{\nu}}\Gamma({\nu}+2)\frac{3{\nu}^2+2{\nu}-5}{3({\nu}+2)({\nu}+3)}\\
\mbox{and}\\
c_3&=\left(\frac{a_4}{256}-\frac{a^2_2}{512}\right)\int_0^{\infty}\exp\left(\frac{-t^2}{4({\nu}+1)}\right)t^8t^{2{\nu}-1}dt+\frac{a_2a_3}{1024}\int_0^{\infty}\exp\left(\frac{-t^2}{4({\nu}+1)}\right)t^{10}t^{2{\nu}-1}dt\\
&\hspace{6.5cm}+\frac{a_2^3}{24576}\int_0^{\infty}\exp\left(\frac{-t^2}{4({\nu}+1)}\right)t^{12}t^{2{\nu}-1}dt\\
&\hspace{6.5cm}=-4^{{\nu}-2} ({\nu}+1)^{{\nu}+1}\Gamma({\nu}+2)\frac{{\nu}^3-{\nu}^2-4{\nu}-8}{6({\nu}+2)^2({\nu}+4)}.
\end{align*}

We observe that, when ${\nu}=1$, $c_0=4$, so
\begin{align*}
\lim_{n\longrightarrow \infty}n \int_0^{\infty}\left(\frac{2|J_1(t)|}{t}\right)^ntdt=4,
\end{align*}
which means the maximum value of $I_1(n)$ occurs at $n=2$ and in the limit as $n$ approaches infinity.

However, when $\nu>1$ and $n\geq 2$, the $c_0$ in (\ref{9}) is greater than the $I_{\nu}(n)$. In particular,
\begin{align*}
I_{\nu}(2)=2^{3\nu}\Gamma(\nu+1)^2\int_0^{\infty}\frac{J_{\nu}(t)^2}{t}dt=2^{3\nu-1}\nu !(\nu-1)!< 2^{2\nu-1}(\nu+1)^{\nu}(\nu-1)!=c_0.
\end{align*}

\bigskip

\textbf{Acknowledgment.} We would like to thank H.~K$\ddot{\textmd{o}}$nig for pointing out the general Ball's integral inequality in \cite{OP:2000} and for his many helpful comments.

\appendix
\section{}
The polynomial in (\ref{6}) corresponding to $m=7$ is
\begin{align*}
&1-\frac{1}{180 \, n}t^4-\frac{1}{2835 \, n^2}t^6+\left(\frac{1}{64800 \, n^2}-\frac{1}{37800 \, n^3}\right)t^8+\left(\frac{1}{510300\, n^3}-\frac{1}{467775 \, n^4}\right)t^{10}\\
&+ \Bigg(-\frac{1}{34992000\,  n^3}+\frac{269}{1285956000\,  n^4}-\frac{691}{3831077250 \, n^5}\Bigg)t^{12}
+\Bigg(-\frac{1}{183708000\,  n^4}+\frac{1}{47151720\,  n^5}\\
&-\frac{2}{127702575\,  n^6}\Bigg)t^{14}+\Bigg(\frac{1}{25194240000\,  n^4}-\frac{349}{462944160000\,  n^5}+\frac{23237}{11033502480000 \, n^6}\\
&-\frac{3617}{2605132530000\,  n^7}\Bigg)t^{16}+\Bigg(\frac{1}{99202320000\,  n^5}-\frac{5543}{60153806790000\,  n^6}+\frac{9001}{43444416015000 \, n^7}\\
&-\frac{43867}{350813659321125\,  n^8}\Bigg)t^{18}+\Bigg(-\frac{1}{22674816000000\,  n^5}+\frac{143}{83329948800000\,  n^6}-\frac{146843}{13902213124800000 \, n^7}\\
&+\frac{62809}{3094897445640000\,  n^8}-\frac{174611}{15313294652906250 \, n^9}\Bigg)t^{20}+\Bigg(\frac{-1}{71425670400000\,  n^6}+\frac{10643}{43310740888800000\,  n^7}\\
&-\frac{17}{14582741040000 \, n^8}+\frac{1621577}{817189465242150000\,  n^9}-\frac{155366}{147926426347074375 \, n^{10}}\Bigg)t^{22}\\
&+\Bigg(\frac{1}{24488801280000000\,  n^6}-\frac{509}{179992689408000000\,  n^7}+\frac{90749797}{2837719743034176000000 \, n^8}\\
&-\frac{370206979}{2948075510818838400000\,  n^9}+\frac{441301082837}{2275545784913290890000000 \, n^{10}}-\frac{236364091}{2423034863565078262500\, n^{11}}\Bigg)t^{24}
\end{align*}
\begin{align*}
&+\Bigg(\frac{1}{64283103360000000\,  n^7}-\frac{7241}{155918667199680000000\,  n^8}+\frac{463523}{118238322626424000000\, n^9}\\
&-\frac{465818341}{35008396690973706000000\,  n^{10}}+\frac{41342265857}{2180731377208570436250000\, n^{11}}-\frac{1315862}{144228265688397515625\, n^{12}}\Bigg)t^{26}\\
&+\Bigg(-\frac{1}{30855889612800000000\,  n^7}+\frac{589}{161993420467200000000\,  n^8}-\frac{6915119}{102157910749230336000000\, n^9}\\
&+\frac{3673793561}{7959803879210863680000000\,  n^{10}}-\frac{570787478291}{4095982412843923600200000000 \, n^{11}}\\
&+\frac{27997256387}{15097371072982410712500000\, n^{12}}-\frac{3392780147}{3952575621190533915703125\, n^{13}}\Bigg)t^{28}.
\end{align*}

\bigskip


\end{document}